\documentclass[amscd,amssymb,verbatim,11pt]{amsart}

\usepackage[usenames,dvipsnames,svgnames,table]{xcolor}
\usepackage[pagebackref,linktocpage=true,colorlinks=true,linkcolor=Blue,citecolor=BrickRed,urlcolor=RoyalBlue]{hyperref}
\usepackage[alphabetic,msc-links,abbrev]{amsrefs} 

\usepackage{color}
\usepackage{graphics}
\usepackage[dvips]{graphicx}
\usepackage{wrapfig}
\usepackage{fancyhdr}
\usepackage{latexsym}
\usepackage{mathrsfs}
\usepackage{amsmath}
\usepackage{amssymb}
\usepackage[all]{xy}
\usepackage{fancyhdr}
\usepackage{fancybox}
\usepackage{rotating}
\usepackage[colorinlistoftodos,prependcaption,textsize=tiny]{todonotes}

\usepackage{upgreek}

\author[Aram L. Karakhanyan]{Aram L. Karakhanyan}
\title[Monotonicity Formula]{A new monotonicity formula for quasilinear elliptic free boundary problems}
\address{School of Mathematics, The University of Edinburgh, EH9 3FD}
\email{aram6k@gmail.com}
\thanks{$2000$ {\it Mathematics Subject Classification.\/} Primary
 35R35, 35K05, 35K55.}
\date{}

\theoremstyle{plain}
\newtheorem{theorem}{Theorem}[section]
\newtheorem{lemma}[theorem]{Lemma}

\newtheorem{prop}[theorem]{Proposition}

\theoremstyle{remark}
\newtheorem{remark}[theorem]{Remark}
\theoremstyle{definition}
\newtheorem{defn}[theorem]{Definition}

\renewcommand\epsilon\varepsilon 
\renewcommand\phi\varphi 
\renewcommand\div{\operatorname{div}} 
\newcommand\R{\mathbb{R}} 
\newcommand\Rn{\R^n} 
\numberwithin{equation}{section}

\newcommand\I[1]{\chi_{\{#1\}}}  
\newcommand\fb[1]{\partial{\{#1>0\}}}  

\newcommand\na{\nabla}

\renewcommand\H{\mathscr H}
\renewcommand\L{\rule[-0.5mm]{0.5pt}{4mm}\rule[-0.5mm]{4mm}{0.5pt}}
\usepackage{esint}

\begin{document}

\begin{abstract}
We construct a monotonicity formula for a class of 
free boundary problems associated with the stationary points of the functional 
\[
J(u)=\int_\Omega F(|\nabla u|^2)+\mbox{meas}(\{u>0\}\cap \Omega), 
\]
where $F$ is a density function satisfying some structural conditions.

The onus of proof lies with the careful 
analysis of the ghost function, the gradient part in the  Helmholtz-W\'eyl decomposition of a nonlinear flux that appears in the domain variation formula for $J(u)$. 

As an application we prove full regularity for a class of quasilinear Bernoulli type 
free boundary problems in $\R^3$.

\end{abstract}

\maketitle

%
%


\section{Introduction}
The aim of this paper is to construct  monotonicity 
formulas for nonlinear elliptic free boundary problems.
For linear problems such formulas are well known and had been 
introduced decades ago \cite{ACF},  \cite{APh},  \cite{Spruck}, 
\cite{W-cpde}. They served as powerful tools to classify the 
global profiles and prove fine and partial regularity results 
reminiscent to the classification of the singular points of  minimal surfaces \cite{Simon}. 
There is a strong connection between these two theories 
as one can see in $\R^3$, where the free boundary cones 
can be transformed into minimal surfaces in  $\R^3$ through the 
Minkowski support function \cite{K-catenoid}. 
In particular, the famous double cone \cite{AC}, which is a 
stationary point of the Alt-Caffarelli functional becomes a 
catenoid through this transformation \cite{K-catenoid}.

Since the discovery of these powerful tools the remaining open question is: \textit{Can one construct monotonicity formulas for nonlinear problems?}

In this paper we give affirmative answer to this question for a class of quasilinear 
free boundary problems considered in \cite{ACF}, which amounts to 
studying the stationary points and local minima of the functional 
\begin{eqnarray}\nonumber
 J(u)=\int_{\Omega}F(\nabla u)
+
\lambda(u),
\end{eqnarray}
where $\Omega\subset \R^n$, $\lambda(u)$ is some function of $u$ and 
$F(t):\R_+\mapsto \R_+, F(t)\in C^{2, 1}[0, \infty]$ satisfies the following conditions:
\begin{itemize}
 \item[$(\bf F_1)$] \quad $F(0)=0,\qquad  c_0\leq F'(t)\leq C_0, $ 
\item[$(\bf F_2)$] \quad $\displaystyle 0\leq F''(t)\leq \frac{C_0}{1+t}, $ 
\end{itemize}
for some positive constants $c_0, C_0$, see  \cite{ACF-Quasi} page 2. 
A typical example is $\lambda(u)=\lambda \I {u>0}$, where $\lambda$
is a positive constants and $\chi _D$ is the characteristic function of a set $D$. 
Note that formally $\div(F'(|\nabla u|^2)\nabla u)$
is a measure supported on the boundary of the set 
$\partial\{u>0\}$.

Let $0\in \fb{u}$. We show that there is 
a function $\upphi$ such that 
\begin{align}
A(r):=  \frac 1 {r^n}\int _{B_r} F(|\na u|^2)+\lambda \I {u>0}
-\frac1{r^{n-1}} \int_{\partial B_r}\upphi
\end{align}
is monotone non decreasing function of $r$. Moreover, $A(r)$ is constant 
if and only if $u$ is a homogenous function of degree one. 

The principal difficulty here is
the control of the integral averages of  
$\upphi$. Note that  $\upphi(x)=|x|^{-2}u^2(x)$ if $F(t)=t$, in other words,  when 
the local minimizers are harmonic function in $\{u>0\}$ then $\upphi$ 
is bounded in view of the local Lipschitz regularity of $u$.
If $F$ is not linear then $\upphi=C\frac{u^2(x)}{|x|^2}+\phi$, where $C$ is a constant, and one has to prove that the integral average of the function $\phi$, called  the ghost function,  is bounded as $r\to 0$. 

Our main result is the following 

\begin{theorem}

Suppose $u$ is a local minimizer of $J$. 
Let $B_R(x_0)\subset\Omega\subset \R^n$ and $x_0\in\partial\{u>0\}$. Assume that
$u$ is Lipschitz in $B_R(x_0)$.
Then for $r\leq R$ 
\setlength\arraycolsep{2pt}
\begin{eqnarray}
\\
A(x_0, r)= \frac 1 {r^n}\int _{B_r(x_0)} F(|\na u|^2)+\lambda(u)
-\frac{F'(1)}{r^{n+1}} \int_{\partial B_r(x_0)}u^2
-\frac1{r^{n-1}} \int_{\partial B_r(x_0)}\phi\nonumber
\end{eqnarray}
is a monotone function of $r$. Moreover

\begin{equation}
A'(x_0, r)= \frac 2{r^n}\int_{\partial B_r(x_0)}F'(|\na u|^2)\left(u_\nu-\frac ur\right)^2\ge 0, 
\end{equation}

\begin{itemize}
\item if $\phi^{x_0}$ satisfies the condition \eqref{Btch} at $x_0$, and $u_0$
is a blow up corresponding to the sequence $r_k$, then
 any blow-up limit of
$\displaystyle  u_r(y)=\frac{u_0(r y)}{r}$ is homogeneous function of degree 1.
\item if $\Omega\subset \R^3$ and \eqref{Btch} and \eqref{ffrrtt} are satisfied 
then $x_0$ is a regular free boundary point of $\fb u$. 
\end{itemize}
\end{theorem}

%
%
\vspace{1cm}
\section{Domain variation}

In his section we derive a formula for the first domain
variation of the functional
\begin{eqnarray}\nonumber
 J(u)=\int_{\Omega}F(\nabla u)
+
\lambda(u)
\end{eqnarray}
where $\lambda(u)$ is some function of $u$ and 
$F(t):\R_+\mapsto \R_+, F(t)\in C^{2, 1}[0, \infty]$ satisfying the following conditions:
\begin{itemize}
 \item $F(0)=0,\qquad  c_0\leq F'(t)\leq C_0, $ 
\item $\displaystyle 0\leq F''(t)\leq \frac{C_0}{1+t}, $ 
\end{itemize}
for some positive constants $c_0, C_0$, see  \cite{ACF-Quasi} page 2. 
More generally, one may take $\lambda(u)=\lambda u^q\I {u>0}$, where $\lambda, q$
are positive constants and $\chi_D$ is the characteristic function of a set $D$.

\subsection{Variational Solutions}
Let us consider the following function 
\begin{equation}
\lambda(u)=\lambda_1u^{q}\I{u>0}+\lambda_2(-u)^{q}\I{u\le0}, \qquad q\in(0, 2)\
\end{equation}
where $\lambda_1, \lambda_2$ are positive constants.
The Euler-Lagrange equation for the minima of $J(u)$ has the form 
\begin{align}\label{E-L-F}
 \div(DF(\na u))=q\lambda_1u^{q-1}\I{u>0}-q\lambda_2(-u)^{q-1}\I{u\le 0}, \qquad q\in(0, 2)\\\nonumber
\end{align}
and, moreover for $ q=0$, 
\begin{eqnarray}\label{E-L-p}
&\div(F'(|\nabla u|^2)\nabla u)=0\ \mbox{in}\  \Omega\cap\left(\{u>0\}\cup \{u<0\}\right)\\\nonumber
 & \Psi(|\na u^+|^2)-\Psi(|\na u^-|^2)=\lambda_1-\lambda_2\ \mbox{on}\   \Omega\cap \left(\partial \{u>0\}\cup \partial \{u<0\}\right), 
\end{eqnarray}
where 
\[
\Psi(|\nabla u|^2)= 2|\nabla u|^2F'(|\nabla u|^2)-F(|\nabla u|^2).
\]
If $q\geq 1$ then the right hand side of the equation \eqref{E-L-F} is bounded, and one can show that any bounded weak solution $u$ is 
locally $C^{1, \beta}$ regular. 

Throughout this paper we assume that $\lambda$ is chosen so that $|\nabla u|=1$ on the regular portion of the free boundary for the one phase Bernoulli type problem studied in \cite{ACF-Quasi}. 

We would be concerned with a particular class of generalized solutions  of (\ref{E-L-F})-(\ref{E-L-p}) called  
\textit{variational solutions}. 
 
\begin{defn}\label{defn-v-sol}
Let $q\in[0, p)$ then  
$u\in W^{1, p}(\Omega)$ is said to be a variational solution of (\ref{E-L-F}) if it satisfies (\ref{E-L-F}) (or \eqref{E-L-p} for $q=0$) in weak sense and 
for any $\phi\in C^{0,1}_0(\Omega;\R^n)$
\begin{align}\label{stat-point}
 \int_\Omega\left\{\frac{\partial F}{\partial \xi}(\nabla
u)\cdot \na u(y)\na \phi(y)- \big[F(\nabla
u(y))+\lambda(u)\big]\div\phi\right\}dy=0.
\end{align}
Similarly $u$ is said to be a variational solution of (\ref{E-L-p}) 
if $u$ satisfies (\ref{E-L-p}) in weak sense and (\ref{stat-point}). 
\end{defn}

\subsection{First variation}
First we observe  that every minimizer is also a variational solution as the following  computation shows.

\begin{lemma}
Every minimizer is also a variational solution.
\end{lemma}
\begin{proof}
For $\phi\in C^{0,1}_0(\Omega:\R^n)$ we put  $u_t(x)=u(x+t\phi(x))$,
with small $t\in\R$. Then
$\phi_t(x)=x+t\phi(x)$ maps $\Omega$ into itself. After change of variables $y=x+t\phi(x)$ we infer

\begin{align}\label{J-t}
\lefteqn{\int_\Omega \bigg[F(\nabla u_t(x))+\lambda(u_t(x))\bigg]dx=}\\\nonumber &
&=\int_\Omega\bigg[F(\nabla u_t(\phi^{-1}_t(y)))+
\lambda(u(y))\bigg]\bigg[1-t\div(\phi(\phi^{-1}_t(y)))+o(t)]\bigg]dy.\nonumber
\end{align}
Here we used the inverse mapping theorem for $\phi_t:x\mapsto y$, in particular a well-known identity
\begin{eqnarray}
\left|\frac{D(x_1,\dots,x_n)}{D(y_1,\dots,y_n)}\right|=
\left|{\frac{D(y_1,\dots,y_n)}{D(x_1,\dots,x_n)}}\right|^{-1}=\frac{1}{1+t\div\phi+o(t)}.\nonumber
\end{eqnarray}
One can easily verify that 
\begin{eqnarray}\nonumber
 \na u_{t}(x)=\na u(\phi_t(x))\left\{\mathbb I+ t\na \phi(x)\right\}
\end{eqnarray}
 with $\mathbb I$ being the unit matrix, hence

\begin{eqnarray}\nonumber
 \nabla u_t(\phi^{-1}_t(y))=\na u(y)\left\{\mathbb I+ t\na \phi(\phi^{-1}_t(y))\right\}.
\end{eqnarray}

This in conjunction with (\ref{J-t}) yields 

\begin{eqnarray}\nonumber
\int_\Omega\Bigg\{\left[\frac{\partial F}{\partial \xi}\big(\nabla
u\left\{\mathbb I+ t\na \phi(\phi_t^{-1}(y))\right\} \big)\cdot \na u(y)\na \phi(\phi^{-1}_t(y))\right]\bigg[1-t\div(\phi(\phi^{-1}_t(y)))+o(t)]\bigg]\\\nonumber
-\bigg[F(\nabla u_t(\phi^{-1}_t(y)))+
\lambda(u(y))\bigg]\bigg[\div(\phi(\phi^{-1}_t(y)))+o(1)]\bigg]\Bigg\}dy\longrightarrow\\\nonumber
\longrightarrow \int_\Omega\left\{\frac{\partial F}{\partial \xi}(\nabla
u)\cdot \na u(y)\na \phi(y)- \big[F(\nabla
u(y))+\lambda(u)\big]\div\phi\right\}dy=0.
\end{eqnarray}
\end{proof}

\subsection{The case $\lambda(u)=\lambda \I {u>0}$}The main result of this section is the following 
\begin{prop}
Let $u$ be a variational solution, and $\lambda(u)=\lambda \I {u>0}$. Then 
\begin{align}\label{mainid}
 \frac d{dr}\left\{\frac 1 {r^n}\int _{B_r} F(|\na u|^2)+\lambda(u)\right\}
 =
 \frac 2{r^n}\int_{\partial B_r}F'(|\na u|^2)\left(u_\nu-\frac ur\right)^2\\\nonumber
 +
\frac 2{r^{n-1}}\int_{\partial B_r}F'(|\na u|^2)\frac {u}{r^2}(u_\nu-\frac ur)\\\nonumber
=
\frac 2{r^n}\int_{\partial B_r}F'(|\na u|^2)\left(u_\nu-\frac ur\right)^2+\\\nonumber
 +
 F_0\frac d{dr}\left\{ \frac1{r^{n+1}} \int_{\partial B_r}u^2\right\}\\\nonumber
 +
T(r), 
\end{align}
where $F_0$ is a constant, $\nu$ is the outer unit normal and 
\begin{equation}
T(r)=\frac 2{r^{n-1}}\int_{\partial B_r}(F'(|\na u|^2)-F_0)\frac {u}{r^2}(u_\nu-\frac ur).
\end{equation}
\end{prop}

\begin{proof}
For 
\begin{displaymath}
\phi(x)=\left\{ \begin{array}{lll} 
                 x(r-|x|)\ & {\rm for}\ x\in B_r,\\
0\ & {\rm otherwise},
                \end{array}\right.
\end{displaymath}
we have that $\partial_{x_j}\phi^i(x)=\delta_{ij}(r-|x|)-\frac{x_ix_j}{|x|}$. 
Hence \eqref{stat-point} becomes 

\begin{align*}
0=\int_{B_r} 2 F'(|\nabla u|^2)\nabla u  \cdot \left(\nabla u \left[\mathbb I(r-|x|)-\frac{x\otimes x}{|x|}\right]\right)\\ 
-\int_{B_r}(F(|\nabla u|^2)+\lambda(u))(n(r-|x|)-|x|).
\end{align*}
Differentiating  the above identity in $r$, we get for a.e. $r$
\begin{eqnarray*}
 n\int_{B_r} F(|\na u|^2)+\lambda(u)-r\int_{\partial B_r} F(|\na u|^2)+\lambda(u)
 =\\
-\int_{\partial B_r}(2F'(|\na u|^2)\na u\cdot x)\frac{\na u\cdot x} r
+
\int_{B_r} 2F'(|\na u|^2)|\na u|^2.
\end{eqnarray*}
After multiplying both sides  by $r^{-n-1}$, and using 
$\div (F'(|\nabla u|^2)\nabla u)=\lambda^*\H^{n-1}\L\partial_{\rm{red}}\{u>0\}$ \cite{ACF-Quasi}
we obtain from the previous identity, 

\begin{align}\label{mainid}
 \frac d{dr}\left\{\frac 1 {r^n}\int _{B_r} F(|\na u|^2)+\lambda(u)\right\}
 =
 \frac 2{r^n}\int_{\partial B_r}F'(|\na u|^2)\left(u_\nu-\frac ur\right)^2\\\nonumber
 +
\frac 2{r^{n-1}}\int_{\partial B_r}F'(|\na u|^2)\frac {u}{r^2}(u_\nu-\frac ur)\\\nonumber
=
\frac 2{r^n}\int_{\partial B_r}F'(|\na u|^2)\left(u_\nu-\frac ur\right)^2+\\\nonumber
 +
 F_0\frac d{dr}\left\{ \frac1{r^{n+1}} \int_{\partial B_r}u^2\right\}\\\nonumber
 +
T(r), 
\end{align}
where $F_0$ is a constant, $\nu$ is the outer unit normal and 
\begin{equation}
T(r)=\frac 2{r^{n-1}}\int_{\partial B_r}(F'(|\na u|^2)-F_0)\frac {u}{r^2}(u_\nu-\frac ur)
\end{equation}
is the error term. 
\end{proof}
%
%
\vspace{1cm}
\section{Helmhotz-W\'eyl decomposition and the ghost function}
The purpose of this section is to find a convenient form for the last integral in (\ref{mainid}) by 
using the Helmholtz-W\'eyl decomposition. 
For $z\in \fb u\cap B_1$, consider the vectorfield 
\begin{eqnarray}
\displaystyle U^z(x)=
(F'(|\nabla u(x)|^2)-F_0)\frac{2u(x)}{|x-z|^2} \left(\nabla u(x)-u(x)\frac{x-z}{|x-z|^2}\right),
\end{eqnarray}
where $F_0$ is some fixed constant.
Then the error term in \eqref{mainid}
can be written as 
\begin{equation}
T(r)=\frac1{r^{n-1}}\int_{\partial B_r} U^z(x)\cdot \nu.
\end{equation}

For $s\in(1, n)$ we have that 
\begin{eqnarray}\label{U-est}
 &&\int_{B_R(z)}|U^z(x)|^sdx=\\\nonumber
 &&\int_{B_R}\left|(F'(|\nabla u(x+z)|^2)-F_0)\frac{u(x+z)}{|x|^2} \left(\nabla u(x+z)-u(x+z)\frac{x}{|x|^2}\right)\right|^sdx\\\nonumber
&&\leq  C\int_{B_R}\frac{dx}{|x|^s}\\\nonumber
&&=C\frac{R^{n-s}}{n-s},
\end{eqnarray}
with $C$ depending on $\|\nabla u \|_{\infty}$.

Thus $U^z(x)\in L^s(B_R), \forall s\in (1, n)$. Recall the Helmholtz-W\'eyl decomposition for $L^s(D)$, with Lipschitz $\partial D$,  see \cite{FM}
\begin{eqnarray}\label{Helm-1}
 L^s(D)=\Theta(D)\bigoplus H(D), 
\end{eqnarray}
where $\Theta(D)=\{V\in L^s(D):V=\na \phi, \theta\in W^{1, s}(D)\}, H(D)=\{h\in L^s(D): \div h=0, h\cdot \nu=0\}$,
and $\nu$ is the unit outward normal to $\partial D$.
Furthermore if  $V\in L^s(D)$ is decomposed according to (\ref{Helm-1}) $V=\na\phi+h$ by estimates (3.2) and (3.6) of \cite{FM} it follows
that there exists a constant $C_0$ depending only on dimension and the domain $D$ such that 
\begin{eqnarray}\label{helm-2}
 \|\phi\|_{W^{1,s}(D)}\leq C_0 \| V\|_{L^s(D)}.
\end{eqnarray}

Note that $\phi$ is a weak solution to 
the Neumann problem 
\begin{equation}
\Delta \phi=\div U^z, \quad \nabla \phi\cdot\nu |_{\partial D}=U^z\cdot \nu.
\end{equation}

Taking $D=B_1(z), $ we have 
\[
U^z(x)=\na \phi^z(x)+h^z(x)
\]
 for some $\phi^z\in W^{1, s}(B_1), h^z\in H(B_1)$. 

\begin{defn}
$\phi^z(x)$ is called the the ghost function at the free boundary point $z$.\end{defn}

Using the decomposition for $U^z$ we can compute for $0<r<1$

\begin{eqnarray}
 \frac 1{r^{n-1}}\int_{\partial B_r(z)}U^z(x)\cdot \nu&=&\frac 1{r^{n-1}}\int_{\partial B_r(z)}(\na \phi^z+h^z)\cdot \nu\\\nonumber
&=&\frac 1{r^{n-1}} \int_{\partial B_r(z)}\na \phi^z \cdot \nu +\frac 1{r^{n-1}}\int_{B_r(z)}\div h^z\\\nonumber
&=&\frac d{dr}\left(\frac 1{r^{n-1}}\int_{\partial B_r(z)}\phi^z\right).
\end{eqnarray}

Moreover,
\begin{equation}\label{Raj}
\frac d{dr}\fint_{B_r(z)}\phi^z=\fint_{B_r(z)}U^z(x)\cdot\frac{x-z}{|x-z|}dx. 
\end{equation}

The crucial property is the following scaling invariance that preserves the possible logarithmic singularity of $\phi^z$:
if $\phi^z$ is the ghost function for $U^z$ for some $z\in \fb u\cap B_1$, then 
the scaled function $\phi_\theta^z(x)=\phi^z(z+\theta x), \theta>0$ solves 
$\Delta \phi_\theta^z=\div U_\theta^z$ where 
\begin{align}\label{U-scale}
U_\theta^z(x)=\frac1{|x|}
(F'(|\nabla u(z+\theta x)|)-F_0)\frac{u(z+\theta x)}{|\theta x|} \left(\nabla u(z+\theta x)-u(z+\theta x)\frac{\theta x}{|\theta x|^2}\right), 
\end{align}
and consequently, in view of Lipschitz regularity  of $u$
\begin{equation}
|U_\theta^z(x)|\sim \frac C{|x-z|}
\end{equation}
near the free boundary point $z$ with $C$ being independent of $\theta$.

Returning to (\ref{mainid})  and introducing 

\begin{equation}\label{A-defn}
A(z, r)= \frac 1 {r^n}\int _{B_r(z)} F(|\na u|^2)+\lambda(u)
-\frac{F_0}{r^{n+1}} \int_{\partial B_r(z)}u^2
-\frac1{r^{n-1}} \int_{\partial B_r(z)}\phi^z
\end{equation}
we see that 
\begin{equation}
A'(z, r)= \frac 2{r^n}\int_{\partial B_r(z)}F'(|\na u|^2)\left(u_\nu-\frac ur\right)^2\ge 0. 
\end{equation}

%
%
\section{The properties of $\phi$}
In this section we prove some uniform estimates for $\phi$
that will be used to control $A(z, r)$ as $r\downarrow 0.$

\subsection{A BMO estimate}
\begin{lemma}\label{lem-lem}
 Let $\phi^z$ be defined as above. Then for any $z\in \fb u \cap B_R, B_{2R}\subset \Omega$
 there is a constant $c_{\rm  {\scalebox{0.6}{BMO}}}=c_{\rm  {\scalebox{0.6}{BMO}}}(F, n, \lambda)$ such that 
 \begin{equation}
 \fint_{B_r(z)}|\phi^z-\bar \phi_{r}^z|^2\le c_0, \quad \bar\phi_r^z=\fint_{B_r(z)}\phi^z.
 \end{equation}
\end{lemma}
\begin{proof}
Without loss of generality we may assume $z=0$. It is enough to show that

\begin{align}
 &\int_{B_{k+1}(z)}| \phi -\bar \phi_{k+1} |^2
 \leq \\\nonumber
& \max\left(Cr_{k+1}^n, 
 \frac1{2^{n}}\int_{B_k(z)} |\phi -\bar \phi_{k} |^2, 
 \dots, 
 \frac1{2^{n(\ell+1)}}\int_{B_{k-\ell}(z)} |\phi -\bar \phi_{k-\ell} |^2, 
 \dots, 
  \frac1{2^{n(k+1)}}\int_{B_{1}(z)} |\phi -\bar \phi_{1} |^2
 \right),
\end{align}
for some fixed constant $C$. 
Here $r_k=2^{-k}, B_k(z)=B_{r_k}(z)$.
If the claim fails then there are free boundary points $z_{k_m}\in B_R$, 
solutions $u_m$ to the free boundary problem, and integers $k_m$ such that 
$m\to \infty$ and  
\begin{align}\label{negate}
 &\int_{B_{k+1}(z_m)}| \phi_m -\bar \phi_{m,k_m+1} |^2
 > \\\nonumber
& \max\Bigg(mr_{k_m+1}^n, 
 2^{-n}\int\limits_{B_{k_m}(z_m)} |\phi_m -\bar \phi_{m,k_m} |^2, 
 \dots, 
 2^{-n(k_m+1)}\int\limits_{B_{1}(z_m)} |\phi_m -\bar \phi_{1} |^2
 \Bigg).
\end{align}
For every $m$ as above, let us define the following scaled functions
\[
\Phi_m(x)=\frac{\phi(z_m+r_{k_m}x)}
{\left(\fint\limits_{B_{k_m+1}(z_m)} |\phi_m -\bar \phi_{m,k_m+1} |^2\right)^{\frac12}}.
\]
Then it is easy to see that 
\begin{equation}
\fint_{B_{\frac12}}|\Phi_m-\bar\Phi_{m, \frac12}|^2=1,
\end{equation}
and, for $1\le \ell \le k_m$, it follows from \eqref{negate} that the following estimate holds
\begin{equation}
\fint_{B_{2^\ell}} |\Phi_m-\bar\Phi_{m, 2^\ell} |^2\le 2^n.
\end{equation}
Moreover, \eqref{negate} also implies 
that 

\begin{equation}
\fint_{B_{k_m+1}(z_m)}| \phi_m -\bar \phi_{m,k_m+1} |^2>\frac m{|B_1|}\to \infty.
\end{equation}
From here we infer that 
\begin{equation}
\fint_{B_{2^\ell}} (\Phi_m-\bar\Phi_{m, \frac12})^2\le (\bar\Phi_{m, 2^\ell}-\bar\Phi_{m, \frac12})^2 +2^n.
\end{equation}
From the definition of the ghost function and 
its scaling properties \eqref{U-scale}, it follows from \eqref{Raj} that 
\begin{equation}
\bar\Phi_{m, 2^\ell}= \bar\Phi_{m, \frac12}+O(\frac{\ell^2}{ m^2}),
\end{equation}
which yields 
\begin{equation}
\fint_{B_{2^\ell}} (\Phi_m-\bar\Phi_{m, \frac12})^2
\le 
O(\frac{\ell^2}{ m^2}) +2^n.
\end{equation}

Now introducing $\tilde\Phi_m=\Phi_m-\bar\Phi_{m, \frac12}$, we can summarize its properties as follows 
\begin{eqnarray}
\fint_{B_{\frac12}}\tilde\Phi_m=0, \quad 
\fint_{B_{\frac12}}(\tilde\Phi_m)^2=1,\\\label{ell-m}
\fint_{B_{2^\ell}} (\tilde\Phi_m)^2\le O(\frac{\ell^2}{ m^2})+2^n, \\\label{scaled-PDE}
\Delta \tilde\Phi_m=\div\tilde U_m,
\end{eqnarray}
where 
\begin{equation}
\tilde U_m
=
\frac{U_m}{\left(\fint_{B_{k+1}(z_m)}| \phi_m -\bar \phi_{m,k_m+1} |^2\right)^2}, 
\end{equation}
and 
\begin{equation}
\int_{B_M}|\tilde U_m|^2\le \frac C m M^{n-2}, \quad M>0.
\end{equation}

If we use   $\zeta^2\tilde \Phi_m, \zeta\in C_0^\infty(\R^n)$,  in the weak formulation of 
\eqref{scaled-PDE}, then we conclude after some
simple estimates the Caccioppolli inequality 
\begin{equation}
\int_{}|\nabla \tilde \Phi_m|^2\zeta^2\le C\int_{}|\tilde U_m|^2\zeta^2+(\zeta^2+|\nabla \zeta|^2)\tilde \Phi_m^2
\end{equation}
with some tame constant $C$.
For a suitable choice of the test function $\zeta$, with $\mbox{supp}\zeta\subset B_{2^{\ell+1}}$ this  and \eqref{ell-m} yield
\begin{equation}
\int_{B_{2^\ell}}|\nabla \tilde \Phi_m|^2\le \frac{C2^{\ell(2n-2)}}{m} +C2^{\ell(n-2)}\left(O(\frac{\ell^2}{ m^2})+2^n\right), 
\end{equation}
for fixed $\ell$ and large $m$.

Using a simple compactness argument we see that there is a harmonic function $\tilde\Phi_0(x), x\in \R^n$
such that 
\begin{eqnarray}\label{Zac1}
\fint_{B_{\frac12}}\tilde\Phi_0=0, \\
\label{Zac2}
\fint_{B_{\frac12}}(\tilde\Phi_0)^2=1,\\
\fint_{B_{2^\ell}} (\tilde\Phi_0)^2\le 2^n.
%
\end{eqnarray}
For any $x\in \R^n$
\begin{equation}
|\tilde\Phi_0(x)|=\left|\fint_{B_{|x|}(x)}\tilde\Phi_0\right|
\le 
4^n\fint_{B_{2|x|}(0)}|\tilde\Phi_0|\le 8^n.
\end{equation}
Consequently, Liouville's theorem implies that $\tilde\Phi_0$ is constant 
which is in contradiction with \eqref{Zac1} and \eqref{Zac2}. 
\end{proof}

\subsection{Improvement of BMO estimate}
At a regular free boundary point $z$ we have 
\begin{equation}\label{eq:bunny}
\frac 1{r^{n-1}}\int_{\partial B_r(z)}\phi^z=a_0(z)+a_1(z)r+O(r^2),
\end{equation}
where $a_0(z)$ is finite, see \eqref{mario}. At this point we don't know
whether $a_0(z)$ remains bounded uniformly in $z$. Hence, if we assume 
that for some sequence $z_k$ of regular free boundary points 
$\lim_{k\to\infty} |a_0(z_k)|=\infty$ then there are two scenarios: either 
$z_k\to z_0$ for a subsequence and $z_0$ is a regular free boundary point, or 
$z_0$ is not a regular free boundary point. The first one can be eliminated due to the fact that  
$$
\limsup_{r\to 0}\frac1{r^{n-2}}\int_{B_r(z_0)}|U^{z_0}|^2=0.
$$ 
To rule out the 
second scenario we need to assume that $U^z$ has small norm so that 
there cannot be any singular mass build up at $z_0$ coming from the regular free boundary points. Note that, $a_0(z)$ is constant for all free boundary point $z$ if $F(t)=t$.

%
%
\vspace{1cm}
\section{The main result }
In this section we examine the behavior of the ghost function at 
irregular free boundary points. As we observed 
$\phi^z$ behaves like a VMO function at the regular free boundary points, see \eqref{eq:bunny}.
As the next lemma shows the VMO regularity of $\phi^z$ at an
 irregular free boundary point    $z$ is enough to conclude the homogeneity.  
\begin{lemma}
Let $z=0$ be a free boundary point and 
\begin{equation}\label{Btch}
\lim_{r_k\downarrow 0}\frac1{r^n_k}\int_{B_{r_k}} \left|\phi-\fint_{B_{r_k}}\phi\right|^2=0
\end{equation}
for some sequence $r_k\to 0$. Let $u^o$ be a blow up corresponding to 
$r_k$. Then any blow up of $u^o$ is a homogeneous function of degree one.  

\end{lemma}

\begin{proof}
For a given pair of numbers $b>a>0$ we compute 
\begin{align*}
\frac1{(b r_k)^n}\int_{B_{br_k}}\phi-\frac1{(a r_k)^n}\int_{B_{ar_k}}\phi
=
\int_{ar_k}^{br_k}\frac d{dt}\left\{\frac1{t^n}\int_{B_t}\right\}dt\\
=
\int_{ar_k}^{br_k}\frac1{t^{n+1}}\int_{B_t}\nabla \phi(x)\cdot xdxdt\\
=
\int_{ar_k}^{br_k}\frac1{t^{n+1}}\int_0^t\rho\int_{\partial B_\rho}\nabla \phi(x)\cdot \nu dSd\rho dt\\
=
\int_{ar_k}^{br_k}\frac1{t^{n+1}}\int_0^t\rho\int_{\partial B_\rho}U[u](x)\cdot \nu dSd\rho dt.
\end{align*}
After rescaling $u_k(x)=u(r_kx)/r_k$ we obtain 
\begin{align}\label{mlla}
\frac1{(b r_k)^n}\int_{B_{br_k}}\phi-\frac1{(a r_k)^n}\int_{B_{ar_k}}\phi
=
\int_a^b\frac1{s^{n+1}}\int_{B_s}U[u_k](y)dyds.
\end{align}
Note that 
\begin{align*}
\left|\frac1{(b r_k)^n}\int_{B_{br_k}}\phi-\frac1{(a r_k)^n}\int_{B_{ar_k}}\phi\right|
\le 
\frac1{(r_ka)^n}\int_{B_{r_ka}}\left |\phi-\frac1{(r_kb)^n}\int_{B_{r_kb}}\phi\right|\\
\le
\left(\frac ba\right)^n\frac1{(r_kb)^n}\int_{B_{r_kb}}\left |\phi-\frac1{(r_kb)^n}\int_{B_{r_kb}}\phi\right|.
\end{align*}
Taking $b=1, 0<a<1$ we get the following inequality for the 
scaled functions $\phi_k(x)=\phi(r_kx)$ we have from \eqref{Btch}
\begin{align}\label{blla}
\left| \int_a^1\frac1{s^{n+1}}\int_{B_s}U[u_k](y)dyds\right|
\le 
\left|\int_{B_1}\phi_k-\frac1{a^n}\int_{B_a}\phi_k\right|\\\nonumber
\le \frac1{a^n}\int_{B_{r_k}}\left |\phi-\frac1{r_k^n}\int_{B_{r_k}}\phi\right|\to 0
\end{align}
as $r_k\to 0$. Applying the BMO estimate to the scaled functions
$\tilde \phi=\phi_k(x)-\int_{B_1}\phi_k$ we get the 
uniform estimate 
\begin{equation}\label{Btch2}
\int_{B_1}|\tilde \phi_k|^2\le c_0.
\end{equation}
From here we have that $\tilde \phi_k\to \phi_0$ strongly in $L^q(B_1), 1<q<2$.
Moreover, 
utilizing the scale invariance 
\begin{equation}\label{Btch3}
U[u_{r_k}](x)\to U[u_0](x)
\end{equation}
strongly in $L^1_{loc}(B_1\setminus\{0\})$.
From \eqref{mlla} and \eqref{blla} we infer
\begin{equation}
\int_{B_1}\phi_0=\frac1{a^n}\int_{B_a}\phi_0
\end{equation}
for all $a\in(0, 1)$. 
Therefore, if $u_0$ is a blow up corresponding to the sequence $r_k$, then  
\[
\bar A(u_0, r)=\frac 1 {r^n}\int _{B_r(0)} F(|\na u_0|^2)+\lambda(u_0)
-\frac{F_0}{r^{n+1}} \int_{\partial B_r(0)}u^2_0
\]
is monotone nondecreasing and bounded for $r\in (0, 1).$
From here it the rest runs as  before and we conclude that any blow up of $u_0$ is a homogeneous function of degree one. 
\end{proof}

In view of previous lemma, the remaining case we need to explore further is
\begin{equation}
0<\liminf_{r\downarrow 0} 
\frac1{r^n}\int_{B_{r}(z)} \left|\phi^z-\fint_{B_{r}}\phi^z\right|^2\leq c_0. 
\end{equation}

Under some conditions on $F$, the energy density function, we can show that 
$c_0$ is small, and consequently, $u$ is close to a homogeneous function of degree 
one. Indeed, for $0<\delta<1$, we have from \eqref{mainid}
\begin{align}\label{eq:desmando}
\int_\delta^1\frac1{t^n}\int_{\partial B_t(z)}F'(|\nabla u|^2)\left(u_\nu-\frac u t\right)^2
&=
\int_{\partial B_1(z)}\phi^z-\frac1{\delta^{n-1}}\int_{\partial B_\delta(z)}\phi^z+O(1)\\\nonumber
&=
O(\epsilon^\star\log\frac1\delta), 
\end{align}
where $\epsilon^\star=\sup_{t\in[0, 1]}|F'(t)-F'(1)|$. Note that 
\begin{equation}\label{eq:smallU}
|U^z|\le \frac{\epsilon^\star}{|x-z|}.
\end{equation} 
In fact, as we will see next, this condition alone implies that 
$u$ is close to some homogeneous function of degree one. 
The left hand side of 
\eqref{eq:desmando} can be estimated from below as follows
\begin{align*}
\int_\delta^1 \frac1{t^n}\int_{\partial B_t(z)}F'(|\nabla u|^2)\left(u_\nu-\frac u t\right)^2
&=
\frac1{t^n} \int_{ B_t(z)} F'(|\nabla u|^2)\left(u_\nu-\frac u t\right)^2\Big|_\delta^1
\\
&+
\int_\delta^1 \frac n{t^{n+1}}
\int_{B_t(z)}F'(|\nabla u|^2)\left(u_\nu-\frac u t\right)^2\\
&\ge n\frac1\delta \inf _{t\in [\delta, 1]}\frac1{t^n}\int_{B_t(z)}F'(|\nabla u|^2)\left(u_\nu-\frac u {|x-z|}\right)^2\\
&+ O(1).
\end{align*}

From the last inequality and \eqref{eq:desmando} 
we infer the estimate
\begin{align}\label{Sackss}
\inf _{t\in [\delta, 1]}\frac1{t^n}\int_{B_t(z)}F'(|\nabla u|^2)\left(u_\nu-\frac u {|x-z|}\right)^2
\leq 
\epsilon^\star+O(\frac1{\log\frac1\delta}).
\end{align}
Therefore, if 
$\epsilon^\star $ is sufficiently small then one can check that \eqref{Sackss}
implies that for $\delta\to 0$ there is a blow up of $u$ at $z$, say $u^0$, such that 
\begin{equation}
\inf_{B_1}\left |\frac{ u^0}{|x|}-\nabla u^0\cdot \frac x{|x|}\right|\le \epsilon^\star.
\end{equation}
Thanks to the results from \cite{K}, this last inequality 
also holds in $\R^3$ with infimum replaced by supremum.
More precisely, we have the following 

\begin{lemma}\label{lem:jen}
Let $u$ be a local minimizer in $\R^3$ such that $|\nabla u|\le 1$
and $z\in \fb u$. Suppose \eqref{ffrrtt} is satisfied. 
For any $\delta>0$ there exists $\epsilon >0$ such that 
\begin{equation}
\frac1{r^3}\int_{B_r(z)} |u(x)-\nabla u(x)\cdot (x-z)|\leq \delta r
\end{equation}
whenever 
\begin{equation}\label{eq:pron}
\frac1{r^{}}\int_{B_r(z)}|U^z|^2\le \epsilon.
\end{equation}

\end{lemma}

\begin{proof}
Suppose the assertion of the lemma is false. Then there is 
$\delta_0>0$ and a sequence $r_k\to 0$ so that 
\begin{equation}
\frac1{r_k^{}}\int_{B_{r_k}}|U[u_k](x)|^2=\frac1k\to 0
\end{equation}
for some sequence of minimizers $u_k$, $|\nabla u_k|\le 1$,
but 
\begin{equation}
\frac1{r_k^3}\int_{B_{r_k}(z_k)}|u_k(x)-\nabla u_{k}(x)(x-z_k)|\ge \delta_0r_k.
\end{equation}

Proceeding as in the proof of previous lemma, we can extract a subsequence $r_{k_m}$
such that $\tilde u_{m}(x)=u(z_{k_m}+r_{k_m}x)/r_{k_m}$ converges 
locally uniformly in $C^{0, 1}(\R^3)$ to some local 
minimizer $u_0$ in $\R^3$, such that 
\begin{equation}\label{dashn}
\int_{B_{1}(0)}|u_k(x)-\nabla u_{k}(x)(x-z_k)|^2\ge \delta_0.
\end{equation}
Moreover, $\phi_m(x)=\phi(z_{k_m}+xr_{k_m})\to \phi_0(x)$
strongly in $L^q(B_1), q\in [1, 2)$
and $\phi_0$ has constant means over the balls $B_r(0). $
Therefore we conclude that 
\[
\bar A(u_0, r)=\frac 1 {r^n}\int _{B_r(0)} F(|\na u_0|^2)+\lambda(u_0)
-\frac{F_0}{r^{n+1}} \int_{\partial B_r(0)}u^2_0
\]
is monotone nondecreasing and
\[
\bar A(u_0, R)-\bar A(u_0, r)=\int_r^R\frac1{t^n}\int_{\partial B_t}F'(|\nabla u|^2)\left(\partial_\nu u_0-\frac{u_0}t\right)^2.
\]
Consequently, for $\rho<r<R$
\[
\bar A(u_0, \rho)\le \bar A(u_0, r)\le \bar A(u_0,R).
\]
Using the scale invariance of $\bar A$ we can construct a
blow up limit $u_{00}$ of $u_0$ (by taking $\rho=M\rho_k, \rho_k\downarrow 0, M>0$), and a blow out limit $u_{0\infty}$ of $u_0$  (by taking $R=MR_k, \rho_k\uparrow \infty, M>0$), such that 
the last inequalities imply
 \[
\bar A(u_{00}, M)\le \bar A(u_0, r)\le \bar A(u_{0\infty}, M).
\]
Since $\bar A$ is bounded and monotone then we conclude that 
$u_{00}$ and $u_{0\infty}$ are homogeneous functions of degree one. Hence
using the main result from \cite{K} we conclude that 
$u_{00}$ and $u_{0\infty}$ are half plane solutions, which 
will have the form $x_1^+$ in a suitable coordinate system. This yields 
\[
\bar A(u_{00}, M)=\bar A(u_{0\infty},M)=\frac{|B_1|}2(F'(1)+\lambda-3F_0)
=\bar A(u_0, r)
\]
for every $r>0$. Consequently, $u_0$ is a homogeneous function of degree one.
Again, applying the main theorem from \cite{K} we conclude that 
$u_0$ is a half plane solution, which is in contradiction with \eqref{dashn}. 
\end{proof}
\begin{remark}\label{rem-small}
There is a full class of divergence type quasilinear 
elliptic equations for which the condition \eqref{eq:pron} is satisfied.
Let 
\[
F(t)=t+\alpha\left(t\arctan t-\frac12\log(1+t^2)\right)
\]
for small $\alpha>0$.
Then taking $F_0=1$ we get 
\begin{equation}
|F'(|\nabla u|^2)-1|=\alpha\arctan|\nabla u|^2\le \alpha C(|\nabla u|_\infty).
\end{equation}
Thus for small $\alpha$, the smallness condition on $L^2$ norm of 
the vectorfield $U$ is satisfied with $\epsilon=\alpha C(|\nabla u|_\infty)$.
\end{remark}

\begin{lemma}
Let $u$ be as in Lemma \ref{lem:jen}. If 
\begin{equation}
\frac1{r^3}\int_{B_r(z)} |u(x)-\nabla u(x)\cdot (x-z)|\leq \delta r
\end{equation}
then $u$ is flat in $B_{r/2}$, and hence $\fb u$ is $C^1$ smooth in $B_{\frac r4}$.
\end{lemma}

\begin{proof}
We outline the proof for completeness. 
It is enough to prove the following claim:  for every $\gamma>0$ there is a 
small $\delta$ so that 
\begin{equation}
\int_{B_1} |u(x)-\nabla u(x)\cdot x|\leq \delta 
\end{equation}
implies 
\begin{equation}
\inf_{|e|=1}\sup_{B_{\frac12}}|u-(x\cdot e)^+|\le \gamma.
\end{equation}
Suppose the assertion of the claim is false. Then 
there is a sequence of minimizers $u_k$ in $B_1$ such that 
\begin{equation}
\int_{B_1} |u_k(x)-\nabla u_k(x)\cdot x|=\frac1k\to \infty 
\end{equation}
but 
\begin{equation}\label{eq:mshort}
\inf_{|e|=1}\sup_{B_{\frac12}}|u_k-(x\cdot e)^+|\ge \gamma.
\end{equation}
Proceeding as before, 
we can extract a subsequence $k_m$ such that 
$u_{k_m}\to u_0$ in $C^{0, 1}_{loc}(\R^3)\cap W^{1,2}_{loc}(\R^3)$,
where $u_0$ is a local minimizer in $\R^3$, and 
\begin{equation}
\int_{B_1} |u_0(x)-\nabla u_0(x)\cdot x|=0.
\end{equation}
Consequently, $u_0$ is a homogeneous function of degree one
in $B_1$, hence by \cite{K} $u_0=(e_0\cdot x)^+$ 
for some unit vector $e_0$. Thus from \eqref{eq:mshort} we have 
$\sup_{B_1}|u_k-(e_0\cdot x)^+|\ge \gamma$, and hence from the 
uniform convergence $u_k\to u_0$
we infer that $\sup_{B_1}|u_0-(e_0\cdot x)^+|\ge \gamma$, 
which is a contradiction. 
\end{proof}

Now we summarize our results and  state our main theorem for minimizers.

\begin{theorem}\label{homogeniety}

Suppose $u$ is a variational solution as in Definition \ref{defn-v-sol}. 
Let $B_R(x_0)\subset\Omega\subset \R^n$ and $x_0\in\partial\{u>0\}$. Assume that
$u$ is Lipschitz in $B_R(x_0)$.
Then for $r\leq R$ 
\setlength\arraycolsep{2pt}
\begin{eqnarray}\label{ar}
\\
A(x_0, r)= \frac 1 {r^n}\int _{B_r(x_0)} F(|\na u|^2)+\lambda(u)
-\frac{F_0}{r^{n+1}} \int_{\partial B_r(x_0)}u^2
-\frac1{r^{n-1}} \int_{\partial B_r(x_0)}\phi\nonumber
\end{eqnarray}
is a monotone function of $r$. Moreover

\begin{equation}
A'(x_0, r)= \frac 2{r^n}\int_{\partial B_r(x_0)}F'(|\na u|^2)\left(u_\nu-\frac ur\right)^2\ge 0, 
\end{equation}

\begin{itemize}
\item if $\phi^{x_0}$ satisfies the condition \eqref{Btch} at $x_0$, and $u_0$
is a blow up corresponding to the sequence $r_k$, then
 any blow-up limit of
$\displaystyle  u_r(y)=\frac{u_0(r y)}{r}$ is homogeneous function of degree 1.
\item if $\Omega\subset \R^3$ and \eqref{Btch} and \eqref{ffrrtt} are satisfied 
then $x_0$ is a regular free boundary point of $\fb u$. 
\end{itemize}
\end{theorem}

%
%
\section{Concluding Remarks}
\subsection{The problem in $\R^3$}
In \cite{K} the author proved that if 
\begin{equation}\label{ffrrtt}
1+2\sup \frac{F''(|\nabla u|^2)}{F'(|\nabla u|^2)}
<4
\end{equation}
then the homogeneous minimizers of $J(u)$ 
in $\R^3$ are flat. 
For $F$ as in Remark \ref{rem-small} 
\[
F(t)=t+\alpha\left(t\arctan t-\frac12\log(1+t^2)\right)
\]
we have 
\[
F'(t)=1+\alpha\arctan t, \quad F''(t)=\alpha \frac1{1+t^2}, 
\]
and thus \eqref{ffrrtt} is equivalent to 
\[
\quad  \frac{\alpha}{(1+\alpha\arctan t)(1+t^2)}<\frac32. 
\]
Combining this with Theorem \ref{homogeniety}, we see that 
for small $\alpha$ the local minimizers are smooth in $\R^3$.

\subsection{The two phase problem}

We are given a domain $\Omega\subset\Rn$ and  a function $g\in
W^{1,2}(\Omega),$ which has no vanishing trace on
$\partial\Omega$. We are looking for function $u$ which minimizes
the functional
\begin{equation}\label{minprob}
J(u)=\int_{\Omega}F(|\nabla u|^2)+\lambda_1\I{u>0}+\lambda_2 \I{u\leq0}
\end{equation}
over the class of admissible functions $K_{g}=\{v: v\in
W^{1,p}(\Omega), v-g\in  W^{1,p}_0(\Omega)\}$. Here $\lambda_{1,2}>0$ are
constants. When $J(g)<\infty$ the minimizer
exists and it is locally Lipschitz continuous (cf. \cite{DK}, \cite{Adv}). 
Hence we can generalize Theorem
\ref{homogeniety} to include the two phase problems, stating that $A(z, r)$, see \eqref{A-defn},  is bounded and non decreasing, and 
the blow up of (local) minimizer
$u$ of (\ref{minprob}) is a homogeneous  function of degree one. 

\subsection{The value $A(z, 0^+)$ at regular free boundary points}
To fix the ideas we let 
$z=0$ be the regular free boundary point, $\lambda$ is chosen 
so  that the free boundary condition is $|\nabla u|=1$, and suppose that 
in a small vicinity of $z=0$ we have the expansion 
$\nabla u(x)=e+Mx+\eta (x), \eta(x)=O(|x|^2)$ for $x\in \{u>0\}$. Here $M$ is a symmetric $n\times n$ matrix. 
Then we can compute 
\begin{align*}
|\nabla u(x)|^2-1
&=
(\nabla u(x)-e)\cdot (\nabla u(x)+e)\\
&=
(Mx+\eta(x))\cdot (2 e+Mx +\eta(x))\\
&=
2Mx\cdot e+O(|x|^2).
\end{align*}
Combining this with the Taylor expansion of $F'$ we obtain 
\begin{align*}
F'(|\nabla u(x)|^2)-F'(1)
&= 
F''(1)(|\nabla u(x)|^2-1)+\frac{F'''(1)}{2!}(|\nabla u(x)|^2-1)^2+\dots\\
&=
F''(1)2Mx \cdot e+O(|x|^2).
\end{align*}
Moreover, since 
$u(x)=x\cdot e+\frac12 Mx\cdot x+o(|x|^2)$ in $\{u>0\}\cap B_r$ for small $r>0$, 
it follows that 
\begin{align*}
(F'(|\nabla u(x)|^2)-F'(1))\cdot \nabla\left ( \frac{u^2(x)}{|x|^2}\right)
&=
(F''(1)2Mx \cdot e+Q(x))\cdot \nabla\left ( \frac{u^2(x)}{|x|^2}\right)\\
&=
\nabla\left (F''(1)2Mx\cdot e \frac{u^2(x)}{|x|^2}\right)-Me \left ( \frac{u^2(x)}{|x|^2}\right)\\
&+
O(|x|).
\end{align*}

Note that 
\begin{align*}
\frac{u^2(x)}{|x|^2}
&=
\frac{(x\cdot e+\frac12 Mx\cdot x+o(|x|^2))^2}{|x|^2}\\
&=
\left(e\cdot \frac x{|x|}\right)^2+O(|x|).
\end{align*}

Recalling the definition of the error term $T$ in \eqref{mainid} we see that 
\begin{align*}
T(r)
&=
\frac 2{r^{n-1}}\int_{\partial B_r}(F'(|\na u|^2)-F_0)\frac {u}{r^2}(u_\nu-\frac ur)\\
&=
\frac 1{r^{n-1}}\int_{\partial B_r} \nabla\left (F''(1)2Mx\cdot e \frac{u^2(x)}{|x|^2}\right)\cdot \nu \\
&
-\frac 1{r^{n-1}}\int_{\partial B_r} Me\cdot \nu \left(e\cdot \frac x{|x|}\right)^2+O(r)\\
&=
\frac d{dr}\left( 
\frac 1{r^{n-1}}\int_{\partial B_r} F''(1)2Mx\cdot e \frac{u^2(x)}{|x|^2}
-\frac r{r^{n-1}}\int_{\partial B_r} Me\cdot \nu \left(e\cdot \frac x{|x|}\right)^2+O(r^2)
\right).
\end{align*}
Consequently, since the above computation does not depend on the specific choice of the 
free boundary point, we have at regular  free boundary point $z$ the asymptotic formula 
\begin{equation}\label{mario}
\frac 1{r^{n-1}}\int_{\partial B_r(z)}\phi^z=a_0(z)+a_1(z)r+O(r^2), 
\end{equation}
where $a_0$ and $a_1$ are constants. 
At this point it is not clear how one can  compute the explicit value of $a_0(z)$ at regular free boundary point $z$. 

\subsection{Ghost functions with additional properties}
The proof of Lemma \ref{lem-lem} shows that if $\phi^z=0$ in $\{u=0\}\cap B_1(z)$ (or on sufficiently dense subset of it) then $\phi^z$ is bounded. It seems to be plausible to conjecture that such ghost function exists.   

%
%
\section*{\footnotesize Acknowledgements}
\footnotesize{The author was partially supported by EPSRC grant EP/S03157X/1 Mean curvature measure of free boundary.}


\end{document}